\documentclass[12pt,a4paper]{amsart}

\usepackage{amsmath,amsfonts,amssymb,amsthm,mathrsfs,microtype}
\usepackage{hyperref,graphicx}
\usepackage{color}

\theoremstyle{plain}
\newtheorem{theorem}{Theorem}[section]
\newtheorem{lemma}[theorem]{Lemma}

\newtheorem{proposition}[theorem]{Proposition}
\theoremstyle{definition}

\newtheorem{example}[theorem]{Example}
\newtheorem{remark}[theorem]{Remark}

\DeclareMathOperator{\dimh}{dim_H}

\DeclareMathOperator{\tr}{\mathrm{tr}}
\newcommand{\modulo}{\ (\mathrm{mod}\ 1)}

\allowdisplaybreaks[3]

\begin{document}

\title[On shrinking targets for toral endomorphisms]{On shrinking targets for linear expanding and hyperbolic
  toral endomorphisms}

\author{Zhang-nan Hu}
\address{Z.-N.~Hu, College of Science, China University of Petroleum, Beijing 102249, P. R. China}
\email{hnlgdxhzn@163.com} 

\author{Tomas Persson}
\address{Tomas Persson, Centre for Mathematical Sciences, Lund University, Box 118, 221 00 Lund, Sweden}
\email{tomasp@gmx.com}
	
\author{Wanlou Wu}
\address{Wanlou Wu, School of Mathematics and Statistics, Jiangsu
  Normal University, Xuzhou, 221116, China}
\email{wuwanlou@163.com}
	
\author{Yiwei Zhang}
\address{Yiwei Zhang,
School of Mathematics and Statistics, Center for Mathematical Sciences,
Hubei Key Laboratory of Engineering Modeling and Scientific Computing,
Huazhong University of Science and Technology, Wuhan 430074, China}
\email{yiweizhang@hust.edu.cn}

\date{\today}	
	
\thanks{Zhangnan Hu
  is supported by Science Foundation of China University of
  Petroleum, Beijing (Grant No. 2462023SZBH013) and China
  Postdoctoral Science Foundation (Grant No. 2023M743878).
  Wanlou Wu is supported by NSFC 12001245 and Natural
  Science Foundation of Jiangsu Province BK20201025.
  Yiwei Zhang is partially supported by NSFC 
  Nos. 12161141002  and 12271432.
  \\
  We thank Mikael Persson Sundqvist for helping us with
  Figure~\ref{fig:complicatedellipse}.  }

	
\setlength{\footskip}{18pt}

\begin{abstract}
  Let $A$ be an invertible $d\times d$ matrix with integer
  elements. Then $A$ determines a self-map $T$ of the
  $d$-dimensional torus
  $\mathbb{T}^d=\mathbb{R}^d/\mathbb{Z}^d$. Given a real number
  $\tau>0$, and a sequence $\{z_n\}$ of points in $\mathbb{T}^d$,
  let $W_\tau$ be the set of points $x\in\mathbb{T}^d$ such that
  $T^n(x)\in B(z_n,e^{-n\tau})$ for infinitely many
  $n\in\mathbb{N}$. The Hausdorff dimension of $W_\tau$ has
  previously been studied by Hill--Velani and
  Li--Liao--Velani--Zorin. We provide complete results on the
  Hausdorff dimension of $W_\tau$ for any expanding matrix.  For
  hyperbolic matrices, we compute the dimension of $W_\tau$ only
  when $A$ is a $2 \times 2$ matrix. We give counterexamples to a
  natural candidate for a dimension formula for general dimension
  $d$.
\end{abstract}

\maketitle

\section{Introduction}

Let $T \colon X \to X$ be a measure preserving transformation on
a metric space $X$ which is equipped with an ergodic Borel probability
measure $m$.
For any fixed subset $B\subset X$ of positive measure,
Birkhoff's ergodic theorem implies that
\[
  \{\, x \in X : T^nx\in B \text{ for infinitely many
  }n\in\mathbb{N} \,\}
\]
has $m$-measure 1.  Hill and Velani \cite{HV1995} considered this
set when $B = B(n)$ is a ball that shrinks with time $n$. They
called the points in the set the \emph{well-approximable} points
in analogy with the classical theory of metric Diophantine
approximation \cite{Ca1957,VG1979}, in particular the
Jarn\'{\i}k--Besicovitch theorem \cite{J1929, B1934}, and
introduced the so called \emph{shrinking target problems}: if at
time $n$ one has a ball $B(n) = B (x_0, r_n)$ centred at $x_0$
with radius $r_n \to 0$, then what kind of properties does the
set of points $x$ have, whose images $T^n(x)$ are in $B(n)$ for
infinitely many $n$?

There are plenty of related works such as the so called
quantitative recurrence properties \cite{B1993}, dynamical
Borel--Cantelli lemma \cite{CK2001}, shrinking target problems
\cite{F2006,T2008}, uniform Diophantine approximation
\cite{BL2016}, recurrence time \cite{BS2001}, waiting time
\cite{G2005}, etc. We refer to the survey article by Wang and Wu
\cite{WangWuSurvey} for more information.

Let $A$ be a $d\times d$ matrix with integral coefficients, let
$X$ be the $d$ dimensional torus
$\mathbb{T}^d=\mathbb{R}^d/\mathbb{Z}^d$ and let
$T \colon \mathbb{T}^d\rightarrow\mathbb{T}^d$ be a
transformation of $\mathbb{T}^d$ defined by
\[
  T(x) = A x \mod 1.
\]
Given $\tau>0$, and a sequence $\{\, z_n : n\in\mathbb{N} \,\}$
of points in the $d$ dimensional torus $\mathbb{T}^d$, we
consider the set
\begin{align*}
  W_\tau &= \{\, x \in \mathbb{T}^d : T^n(x)\in B(z_n,e^{-n\tau})
           \text{ for infinitely many }n\in\mathbb{N} \,\} \\
  &=\limsup_{n\in\mathbb{N}}T^{-n}B(z_n,e^{-n\tau}).
\end{align*}

Let $\mu$ denote the $d$-dimensional normalized Lebesgue measure
on $\mathbb{T}^d$. In this paper, the radii of the balls decay
exponentially as $e^{-\tau n}$. Li et al.\ \cite{Lietal} also
studied other decay rates and they gave conditions when the
Lebesgue measure of the corresponding set is $0$ or $1$.  We will
only study the exponential decay rate, and we note that by the
Borel--Cantelli lemma, it follows immediately that for any
$\tau>0$, we have $\mu(W_\tau)=0$.

Since $\mu(W_\tau) = 0$, it is natural to calculate the Hausdorff
dimension of the set $W_\tau$. This is the problem that we will deal
with in this paper. It has previously been studied by Hill and
Velani \cite{HillVelaniMatrix} and by Li, Liao, Velani and Zorin
\cite[Theorem~8 and 9]{Lietal}. Our results are similar to some
of those in the mentioned papers, but our assumptions are
somewhat different, and we also prove a large intersection
property of the set $W_\tau$. The main difficulty in the problem
of determining the Hausdorff dimension of $W_\tau$ is that
$W_\tau$ is a limsup set of ellipsoids that are increasingly more
eccentric.

Li et al.\ used mass transference to obtain their result. We use
the geometry of the involved sets to obtain estimates on measures
which leads to our result through results on Riesz energies.

Our results about hyperbolic endomorphisms
are only for $d = 2$, but we give complete results for any
expanding endomorphism when $d\ge 2$.
We also point out an error in Theorem~1 of Hill and
Velani \cite{HillVelaniMatrix}, see
Example~\ref{ex:errorinHillVelani} below.

\section{Results}

Our results will involve the eigenvalues of the matrix
$A$. Throughout this text, we denote by
$\lambda_1, \ldots,\lambda_d$ the eigenvalues of $A$ counted with
multiplicity and ordered so that
\[
  0 < \lvert \lambda_1 \rvert \leq \lvert \lambda_2 \rvert \leq
  \ldots \leq \lvert \lambda_d \rvert.
\]
Hence, we will always assume that $A$ is invertible.
Put $l_j = \log \lvert \lambda_j \rvert$, $j = 1,2,\ldots,d$.

With this notation, we have the following theorems.

\begin{theorem} \label{the:dimensionofW1} Let $A$ be a
  $2\times 2$ integer matrix with eigenvalues
  $0 < |\lambda_1|<1<|\lambda_2|$. Then 
  \[
  \dimh W_\tau = s_{\tau}
    = \left\{ \begin{array}{ll}
                \dfrac{2 l_2}{\tau + l_2}, & 0 < \tau < -l_1, \\
                \min\biggl\{ \dfrac{l_1 + l_2}{\tau + l_1},
                \dfrac{2 l_2}{\tau + l_2} \biggr\}, &  -l_1 < \tau.
                \rule{0pt}{25pt}
              \end{array}
            \right.         
           \]
             Moreover, $W_\tau$ has large intersection property in the sense
  that $W_\tau \in \mathscr{G}^{s_{\tau}}$. (See
  Section~\ref{hmd} for the definition of\/
  $\mathscr{G}^{s_{\tau}}$).
\end{theorem}

We remark that the dimension formula above differs from the one
for recurrence in Hu and Persson \cite{HuPersson}. In
one-dimensional cases the dimension formulae are often the same
for hitting and recurrence, see \cite{Lietal,WT}.

\begin{remark}
  In Theorem~\ref{the:dimensionofW1}, for the critical case
  $\tau = -l_1$, the dimension of $W_\tau$ if simpler to describe
  when $|\det A| > 1$. When $|\det A|>1$, we have
  $\dimh W_{-l_1} = \frac{2 l_2}{ l_2-l_1}$, and the dimension of
  $W_\tau$ is continuous as a function of $\tau$. We refer to the
  proof of Theorem~\ref{the:dimensionofW1}.

  When $|\det A|=1$ and $\lambda$ is an eigenvalue with
  $|\lambda| > 1$, we have by Theorem~\ref{the:dimensionofW1}
  that
  \[
    \dimh W_\tau = \left\{ \begin{array}{ll} \dfrac{2
          \log|\lambda|}{\tau + \log|\lambda|}, & 0 <
                                                 \tau <  \log|\lambda|, \\
                             0, &  \log|\lambda| < \tau.
                                 \rule{0pt}{15pt}\end{array}
                             \right.  .
  \]
  It is interesting (but natural) that the dimension formula
  above is not continuous as a function of
  $\tau$. Figure~\ref{fig:cat} shows the graph of $\dimh W_\tau$
  as a function of $\tau$ when
  $A= \bigl( \begin{smallmatrix}2 & 1 \\ 1 & 1 \end{smallmatrix}
  \bigr)$.
 
  \begin{figure}
    \begin{center}
      \includegraphics{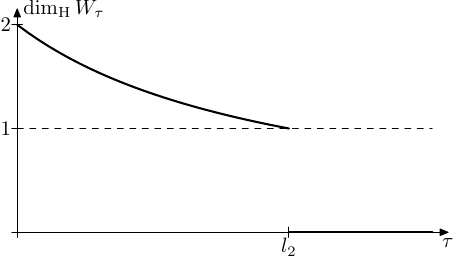}
    \end{center}
    \caption{The graph of $\tau \mapsto \dimh W_\tau$ for the cat
      map, i.e.\ when
      $A= \bigl( \protect\begin{smallmatrix}2 & 1 \protect \\ 1 &
        1 \protect\end{smallmatrix} \bigr)$.} \label{fig:cat}
  \end{figure}

The dimension of $W_\tau$ when $\tau = \log |\lambda|$ depends
  on the choice of $z_n$. It is possible to choose $z_n$ so
  that $W_\tau$ is a line segment, and hence $\dimh W_\tau = 1$,
  and it is possible to choose $z_n$ so that $W_\tau = \emptyset$
  and $\dimh W_\tau = 0$.
\end{remark}

We now turn to the case of a torus of general dimension $d$. We prove an upper bound on the dimension of $W_\tau$.

\begin{theorem}
  \label{upperboundofW}
  If $A$ is an invertible $d\times d$ integer matrix, then for
  any $\tau\ge0$, we have
  \begin{equation} \label{eq:upperbound}
    \dimh W_\tau \leq \min_k \left \{ \dfrac{kl_k + \sum_{j>k}
        l_j}{\tau + l_k} \right \},
  \end{equation}
  where the minimum is over those $k$ for which $\tau+l_k>0$.
\end{theorem}
Hill and Velani \cite[Theorem~1 and 2]{HillVelaniMatrix} state
that under some assumptions, there is equality in
\eqref{eq:upperbound}. As we will show in
Examples~\ref{ex:errorinHillVelani} and \ref{ex:last}, those
statements are incorrect for hyperbolic matrices. These examples
also show that the lower bound mentioned immediately after Hill's
and Velani's Theorem~2, is not always correct. However, the
following theorem shows that their statements hold for expanding
toral endomorphisms.

We further state a complete dimension formula for expanding
endomorphisms. We impose, however, the stronger assumption that
all eigenvalues are outside the unit circle, without which the
result may fail.

\begin{theorem} \label{the:dimensionofW2}
  Let $A$ be a $d\times d$ integer matrix.  Assume that $l_i>0$
  for $1\le i\le d$. Then 
  \[
   \dimh W_{\tau}= \tilde{s}_\tau := \min_{1\le k\le
      d} \Biggl\{ \frac{k l_ k + \sum_{j > k} l_j - \sum_{j = 1}^d
        (l_j - l_k - \tau)_+}{\tau + l_k} \Biggr\},
  \]
  and $(x)_+ = \max \{0, x \}$.
  Moreover, $W_{\tau}\in\mathscr{G}^{\tilde{s}_\tau},$
\end{theorem}

Li et al.\ \cite[Theorem~8 and 9]{Lietal} obtained dimension
formulae under some conditions on the matrix $A$. They assumed
that either $A$ is an integer matrix, diagonalisable over
$\mathbb{Z}$, with all eigenvalues outside the unit circle, or
that $A$ is diagonal, not necessarily an integer matrix, and that
all eigenvalues are outside the unit circle. The novelty in
Theorem~\ref{the:dimensionofW2} is that 
we do not need any assumption on diagonalisability over
$\mathbb{Z}$, but on the other hand, we only treat integer
matrices.

One could possibly expect that the lower bound of
Theorem~\ref{the:dimensionofW2} holds also without any
assumptions on the eigenvalues. Although this might be the case
in certain cases, it is not always so. To explain what can go
wrong, we define a probability measure $\mu_n$ supported on
$E_n = T^{-n} B (z_n, e^{-n\tau})$ by
\[
  \mu_n = c_d e^{nd\tau} \mu|_{E_n},
\]
where $c_d$ is a constant that depends only on $d$. In words,
$\mu_n$ is the normalised restriction of the Lebesgue measure to
the $n$-th inverse image of $B (z_n, e^{-\tau n})$.

It turns out that when $\mu_n$ does not converge weakly to the
Lebesgue measure, then the lower bound in
Theorem~\ref{the:dimensionofW2} can fail. In the proof of
Theorem~\ref{the:dimensionofW1}, we will see (and use) that
$\mu_n$ converges weakly to the Lebesgue measure, unless in the
case when $\dimh W_\tau = 0$. Hence, in the case of
Theorem~\ref{the:dimensionofW1}, we do not need the assumption
that $\mu_n$ converges weakly to the Lebesgue measure, but it
seems that such an assumption is needed in the general case.

If $A$ has an eigenvalue on the unit circle, then it is sometimes
possible to, given any $\tau > 0$, choose $z_n$ such that
$W_\tau = \emptyset$, and $\mu_n$ will not converge to the
Lebesgue measure.  The following example shows a particular
instance of this. In Section~\ref{sec:examples} there are further
examples of things that can go wrong if the assumptions are not
fulfilled. In particular, the examples in
Section~\ref{sec:examples} show that it is not enough to assume
that there are no eigenvalues on the unit circle, in order for the
bound in Theorem~\ref{the:dimensionofW2} to hold.

\begin{example}
  \label{ex:errorinHillVelani}
  Assume that $b$ is an integer, and define the $2\times 2$
  matrix
  \[
    A = \begin{pmatrix}1 & 0 \\ 0 & b \end{pmatrix}.
  \]
  If $z_n=(x_n,y_n)$ with $x_n\in\mathbb{T}$ and
  $y_n\in\mathbb{T}$, then
  \[
    T^{-n}B(z_n,e^{-n\tau}) \subset [x_n-e^{-n\tau},x_n +
    e^{-n\tau}] \times \mathbb{T}.
  \]
  From this it is clear that we can choose $x_n$ such that
  \[
    \limsup_{n\in\mathbb{N}} [x_n - e^{-n\tau}, x_n + e^{-n\tau}]
    = \emptyset.
  \]
  We then also have
  \[
    W_\tau \subset \limsup_{n \in \mathbb{N}} ([x_n - e^{-n\tau},
    x_n + e^{-n\tau}] \times \mathbb{T}) = \emptyset.
  \]
  The dimension formula in \cite{HillVelaniMatrix} is not quite
  correct in this and similar cases, such as when
  \[
    A = \begin{pmatrix}1 & 0 \\ 0 & B \end{pmatrix},
  \]
  and $B$ is any $(d-1)\times(d-1)$ integer matrix.  In
  particular, the main result of Hill and Velani
  \cite[Theorem~1]{HillVelaniMatrix} is not correct as stated,
  see also Example~\ref{ex:last}.
\end{example}

The paper is organised as follows. Section~\ref{hmd} is devoted
to preliminaries, where we recall the definition of the Hausdorff
dimension and give a critical lemma.  In this paper, we deal with
two kinds of endomorphisms, where
Theorems~\ref{the:dimensionofW1} and \ref{upperboundofW} are for
hyperbolic endomorphisms, and Theorem~\ref{the:dimensionofW2} is
for expanding endomorphisms. In Section~\ref{upperbound}, we give
the proofs of Theorem~\ref{upperboundofW} and the upper bound in
Theorem~\ref{the:dimensionofW2}.  The strategy of dealing with
the lower bound for these two kinds of endomorphisms is
different, so the proofs are divided into two
sections. Section~\ref{sec:lowerboundhyperbolic} contains the
proofs of Theorem~\ref{the:dimensionofW1}, and the proof of
Theorem~\ref{the:dimensionofW2} is given in
Section~\ref{sec:lowerbound}.  In the last section, we give
counterexamples to a natural candidate for a dimension formula
for general dimension $d$.

Our method of proof is the following. We give upper bounds on the
Hausdorff dimension by direct covering arguments combined with a
proposition by Wang and Wu \cite{WangWu},
Proposition~\ref{simple}. The lower bounds are obtained by
estimates of the type $\mu_n (B(x,r)) \leq C r^s$, where, as
above, $\mu_n$ is the normalised restriction of the Lebesgue
measure to the set $T^{-n} (B(z_n,e^{-\tau n}))$. According to
Lemma~\ref{lemmaforlb} below, such estimates lead to a lower
bound on the Hausdorff dimension provided that $\mu_n$ converges
weakly to the Lebesgue measure on $\mathbb{T}^d$.

\section{Hausdorff measure and dimension} \label{hmd}

Limsup-sets often possess a large intersection property, see
\cite{F1988, F1994,P2015}. This means that the set belongs to a
particular class $\mathscr{G}^s$ of $G_\delta$ sets that, among
other properties, is closed under countable intersections and
consists of sets of Hausdorff dimension at least $s$.

To get an upper bound on the Hausdorff dimension of a set is
frequently easier than obtaining a lower bound, at least if the
set is a limsup-set. The mass distribution principle
\cite{BeresnevichVelani} is often used to get a lower bound of
the Hausdorff dimension of a set. In this paper, we will use the
following special case of a lemma from Persson and Reeve
\cite{PR2015}. See also Persson \cite{P2015}.  It is a slight
variation of a lemma in \cite{P2015}.

Let $\mu$ denote the $d$-dimensional Lebesgue measure on
$\mathbb{T}^d$. For a set $B\subset\mathbb{T}^d$, we define the
$s$-dimensional Riesz energy of $B$ as
\[
  I_s(B) = \iint_{B\times B}\lvert x - y \rvert^{-s} \,
  \mathrm{d}x \mathrm{d}y,
\]
and the $s$-dimensional Riesz energy of the measure $\mu$ as
\[
  I_s (\mu) = \iint |x-y|^{-s} \, \mathrm{d} \mu (x) \mathrm{d}
  \mu (y).
\]


\begin{lemma}[{\cite[Lemma~2.1]{P2022}}]\label{lemma:persson}
  Let $E_n$ be open sets in $\mathbb{T}^d$ and let $\mu_n$ be
  probability measures with
  $\mu_n( \mathbb{T}^d\setminus E_n) = 0$. If there is a constant
  $C>1$ such that
  \begin{equation}\label{munball}
    C^{-1} \le \liminf_{n\to\infty} \frac{\mu_n(B)}{\mu(B)}\le
    \limsup_{n\to\infty} \frac{\mu_n(B)}{\mu(B)}  \le C
  \end{equation}
  for any ball $B$, and
  \[
    \iint |x-y|^{-s} \mathrm{d} \mu_n(x) \mathrm{d} \mu_n(y)<C
  \]
  for all $n$, then
  $\limsup_{n\to\infty}E_n \in \mathscr{G}^s$, and in
  particular we have
  \[
    \dimh \bigl(\limsup_{n\to\infty}E_n \bigr) \ge s.
  \]
\end{lemma}

Sometimes the assumption that the $s$-dimensional Riesz energy of
$\mu_n$ is uniformly bounded is not easy to check. We observe
that this assumption can be replaced by a stronger but sometimes
more manageable one as the following lemma shows, which is
a slight variation of Lemma~\ref{lemma:persson}.

\begin{lemma} \label{lemmaforlb}
  Let $E_n$ be open sets in $\mathbb{T}^d$ and let $\mu_n$ be
  probability measures with
  $\mu_n(\mathbb{T}^d \setminus E_n) =
  0$.
  Suppose that there are constants $C$ and $s$ such that
  \begin{equation}\label{munballn}
    C^{-1} \le \liminf_{n\to\infty} \frac{\mu_n(B)}{\mu(B)}\le
    \limsup_{n\to\infty} \frac{\mu_n(B)}{\mu(B)}  \le C
  \end{equation}
  for any ball $B$ and $\mu_n(B) \le Cr^s$ for all $n$ and any
  ball $B$ of radius $r$.  Then
  $\limsup_{n\to\infty}
  E_n\in\mathscr{G}^s.$
\end{lemma}

\begin{proof} Let $t < s$. Using the esimate $\mu_n(B) \le Cr^s$,
  we can write
  \begin{align*}
    \int |x-y|^{-t} \mathrm{d} \mu_n(y)
    & =
      1+\int_1^{\infty}\mu_n\bigl(B(x, u^{-1/t})\bigr) \mathrm{d}
      u \le 1+
      2C\int_1^{\infty}u^{-s/t}  \mathrm{d} u\\
    & = 1+2C\frac{t}{s-t}.
  \end{align*}
  Therefore,
  $\iint |x - y|^{-t} \mathrm{d} \mu_n(x) \mathrm{d} \mu_n(y) \le
  1 + 2Ct/(s-t)$ and by Lemma~\ref{lemma:persson} we can conclude
  that $\limsup_{n\to\infty} E_n$ belongs to the intersection
  class with dimension $t$. Since $t$ can be taken as close to
  $s$ as we like, it follows that
  $\limsup_{n\to\infty} E_n\in\mathscr{G}^s.$
\end{proof}

\section{The upper bound}
\label{upperbound}

We start this section with several elementary observations
concerning $T^{-n} (B(z_n, e^{-\tau n}))$.

In the proofs, the singular values of $A^n$ will be of
importance. The following easy lemma tells us that the singular
values and eigenvalues of $A^n$ are comparable.

\begin{lemma}\label{ellipsoid:semiaxes}
  Let $\sigma_{n,1} \leq \ldots \leq \sigma_{n,d}$ be the
  singular values of $A^n$. For any $\varepsilon > 0$, there
  exists a constant $c > 0$ such that
  \[
    c^{-1} e^{-\varepsilon n} \leq \frac{e^{l_k n}}{\sigma_{n,k}}
    \leq c e^{\varepsilon n}.
  \]
  for all $k$ and $n$.
\end{lemma}

We shall also need the following lemma that we compile from the
book by Everest and Ward \cite{ew1999}.

\begin{lemma}[{\cite[Lemma~2.2 and 2.3]{ew1999}}]
  \label{pointnumber}
  If $A$ is an invertible $d \times d$ integer matrix with no
  eigenvalue being a root of unity, and $T (x) = A x \modulo$,
  then
  \[
    \# \{ \, x \in \mathbb{T}^d : A^n x \modulo = 0 \, \} = |\det
    A|^n.
  \]
\end{lemma}
  
Put $L:=\sum_{i=1}^d l_i$, then $e^{nL}=|\det A|^n$, which is an
integer. Let $\tau > 0$. We let $k$ be the smallest number such
that $\tau + l_k > 0$. By Lemma~\ref{pointnumber}, the set
$T^{-n} (B(z_n, e^{-\tau n}))$ consists of $e^{nL}$ ellipsoids
with semi-axes $r_1 \geq r_2 \geq \ldots \geq r_d > 0$, where
$r_i=e^{-(\tau n + \sigma_{n,i})}$, $1\le i\le k$.  It follows
from Lemma~\ref{ellipsoid:semiaxes} that for $n\ge1$ and
$1\le j\le d$,
\[
  c^{-1}e^{-(l_k+\tau+\epsilon)n} < r_j \le
  ce^{-(l_k+\tau-\epsilon)n}.
\]
By the choice of $k$, we have that $r_k, \ldots, r_d$ are
exponentially small (in $n$), whereas $r_j$ for $j \leq k-1$ are
exponentially large.
   
We are now ready to prove Theorem~\ref{upperboundofW}.

\begin{proof}[Proof of Theorem~\ref{upperboundofW}]
  To obtain an upper bound on the Hausdorff dimension of
  $W_\tau$, we look for covers of the set $W_\tau$. Recall that
  \[
    W_\tau = \limsup_{n\in\mathbb{N}} T^{-n} B(z_n,e^{-n\tau}).
  \]
  
  Let $\varepsilon > 0$ and recall $L=\sum l_j$. By
  Lemma~\ref{ellipsoid:semiaxes}, for all large $n$, the set
  $T^{-n}B(z_n,e^{-n\tau})$ consists of $e^{Ln}$ ellipsoids with
  semi-axes not more than $e^{-(l_j+\tau-\varepsilon)n}$. Fix an
  integer $k$ such that $1\leq k\leq d$, we may cover the set
  $T^{-n}B(z_n,e^{-n\tau})$ by balls of radius
  $r=e^{-(l_k+\tau-\varepsilon)n}$. We assume that $k$ is such
  that $l_k+\tau-\varepsilon>0$, and hence also that the radius
  goes to zero as $n$ goes to infinity.

  To cover each ellipsoid we need about
  \[
    \prod_{j<k} \dfrac{e^{-(l_j + \tau - \varepsilon)n}}{r} =
    \prod_{j<k} \dfrac{e^{-(l_j + \tau - \varepsilon)n}}{e^{-(l_k
        + \tau - \varepsilon)n}} = e^{n ( ( k - 1 )
        l_k - \sum_{j<k} l_j )}
  \]
  such balls. Hence, we need in total not more than about
  \[
    e^{n ( L + ( k - 1 ) l_k - \sum_{j<k} l_j )}
  \]
  balls of radius $e^{-(l_k+\tau-\varepsilon)n}$ to cover the set
  $T^{-n}B(z_n,e^{-n\tau})$.
   
  If $l_k+\tau>0$, then the radius $e^{-(l_k+\tau)n}$ goes to $0$
  as $n\rightarrow\infty$ and the cover mentioned above can be
  used to get an upper bound on the Hausdorff dimension of
  $W_\tau$. Hence, by letting $\varepsilon \to 0$, we get that
  \[
    \dimh W_\tau\leq\dfrac{kl_k +
      \sum_{j>k}l_j}{\tau+l_k}
  \]
  provided that $\tau+l_k>0$. This proves the upper bound on the
  Hausdorff dimension of $W_\tau$.
\end{proof}

The proof of Theorem~\ref{the:dimensionofW2} is divided into two
parts. We show the upper bound in this section, and give the
lower bound in Section~\ref{sec:lowerbound}.

\begin{proof}[Proof of the upper bound in
  Theorem~\ref{the:dimensionofW2}]
  
  To obtain the upper bound on the Hausdorff dimension of
  $W_\tau$ in Theorem~\ref{upperboundofW}, we cover the
  ellipsoids individually. Given $1\le k\le d$ with $l_k+\tau>0$,
  let $r=e^{-(l_k+\tau-\epsilon)}$, then the number of balls of
  radius $r$ covering $T^{-n}B(z_n,e^{-\tau n})$ is not more than
  \[
    N= e^{n ( L + ( k - 1 ) l_k - \sum_{j<k} l_j )}.
  \]
  However, as we shall see below, we can sometimes do with much
  fewer balls.

  In the eigenspace of the eigenvalue $\lambda_{j}$, the
  separation between two nearby ellipsoids is not more than
  $ e^{-(l_j-\epsilon)n}$. Let $d_j$ be the dimension of this
  eigenspace. Then, if $r > e^{-(l_j-\epsilon)n}$, every ball of
  radius $r$ used in the cover will intersect at least
  \[
    \Big(\frac{r}{e^{-(l_j-\epsilon)n}}\Big)^{d_j}=e^{-d_j(l_k+\tau-l_j)n}
  \]
  ellipsoids in the directions of the eigenspace. Letting
  $(x)_+ = \max\{0, x\}$, we therefore have that each ball in the
  cover intersects
  \[
    e^{d_j (l_j-l_k-\tau)_+n}
  \]
  ellipsoids in the direction of the eigenspace, no matter if
  $r > e^{-(l_j-\epsilon)n}$ or not.  In this case $d_j=1$,
  $j=1,2,\dots,d$.

  All together, this means that we need not more than about
  \begin{equation*}
    \begin{split}
      N\prod_j e^{-(l_j-l_k-\tau)_+n}&= e^{n(L+(k-1)l_k-\sum_{j<k} l_j-\sum_j (l_j-l_k-\tau)_+)}\\
      &= e^{n(kl_k+\sum_{j>k} l_j-\sum_j (l_j-l_k-\tau)_+)}
    \end{split}
  \end{equation*}
  ellipses to cover $T^{-n}(B(z_n, e^{-\tau n}))$.

  Letting $\epsilon\to 0$, we conclude that
  \[
    \dimh W_\tau \le\frac{kl_k +\sum_{j>k} l_j -\sum_j(l_j - l_k
      - \tau )_+}{\tau + l_k}
  \]
  provided that $\tau +l_k > 0$. Thereby we prove the upper bound
  on $\dimh W_\tau$.
\end{proof}

\section{The proof of Theorem~\ref{the:dimensionofW1}}
\label{sec:lowerboundhyperbolic}

First, let us give some notation.  We write $f_n\lesssim g_n$,
$n\in\mathbb{N}$, if there is an absolute constant $0<c< \infty$
such that $f_n\le cg_n$ for large $n$.  If $f_n\lesssim g_n$ and
$g_n\lesssim f_n$ , then we write $f_n\asymp g_n$.

For the lower bound in our theorems, we use
Lemma~\ref{lemmaforlb}.  In this section, let $\mu$ denote the
$2$-dimensional normalized Lebesgue measure on $\mathbb{T}^2$,
and we define the probability measure $\mu_n$ supported on
$E_n=T^{-n}B(z_n,e^{-n\tau})$ by
\[
  \mu_n=\pi^{-1}e^{n2\tau}\mu\arrowvert_{E_n}.
\]

The following lemma is essential in estimating $\dimh W_{\tau}$
when $A$ is a hyperbolic $2\times 2$ integer matrix.
  
\begin{lemma} \label{lem:algebraic}
  Let $A$ be a $2\times 2$ integer matrix with an eigenvalue
  $0 < |\lambda|<1$. Then both eigenvalues of $M$ are irrational.
\end{lemma}

\begin{proof}
  Suppose $D=\det A$ and $\lambda=p/q$, where $p/q$ is a reduced
  fraction.  Then the other eigenvalue must be $qD/p$ since the
  product of the eigenvalues is $D$. The trace of the matrix is
  also an integer and is the sum of the eigenvalues. Hence
  \[
    p/q + Dq/p = (p^2 + Dq^2)/(pq)
  \]
  is an integer. Therefore, $(p^2 + Dq^2)$ is divisible by
  $q$. But then $q $ divides $p^2$ which is impossible, since
  $p/q$ is a reduced fraction.
\end{proof}

Let $\theta $ be a real irrational number, and let
$\theta - [\theta] =\{\theta\}$ be the fractional part of
$\theta$, where $[\theta]$ is the greatest integer not greater
than $\theta$. The distribution of $(\{\theta n\})_{n\ge1}$ is a
crucial ingredient in showing $(\mu_n)_n$ satisfy inequalities
\eqref{munballn}.

The set $\{\, \{\theta n\} \colon1 \le n\le N \,\}$ partitions
the interval $[0, 1]$ into $N+1$ intervals. The three distance
theorem states that these intervals have at most three distinct
lengths, then we denote $d_\theta (N)$ and $d_\theta'(N)$,
respectively, the maximum and minimum lengths.  The following
theorem describes the relationships between $d_\theta (N)$ and
$d_\theta'(N)$.

\begin{theorem}[Theorem~2 in \cite{MK}]\label{MK}
  Under the setting given above, the sequence
  $(\frac{d_\theta(N)}{d_\theta'(N)})_{N=1}^\infty$ is bounded if
  and only if $\theta $ is of constant type (defined as
  irrationals with bounded partial quotients).
\end{theorem}

\subsection{Proof of Theorem~\ref{the:dimensionofW1} }

Without loss of generality, we assume that $\tau>0$.  Suppose
that $l_1=\log|\lambda_1|$, and $l_2=\log|\lambda_2|$, then
$l_1<0<l_2$.  The assumptions in Theorem~\ref{the:dimensionofW1}
imply that $A$ is diagonalisable and we do not need to use the
$\varepsilon$ in Lemma~\ref{ellipsoid:semiaxes} as in the proof
of Theorem~\ref{upperboundofW} above, and for simplicity we
assume that $c=1$ in Lemma~\ref{ellipsoid:semiaxes}. By
Lemma~\ref{pointnumber}, the set $T^{-n} (B(z_n,e^{- \tau n}))$
consists of $e^{L n}$ ellipses with semi-axes
$e^{- (\tau + l_1)n}$ and $e^{- (\tau + l_2)n}$, where
$L=l_1+l_2\ge0$.

The strategy of the proof is the following. We are going to use
Lemma~\ref{lemmaforlb}, and therefore, we are going to prove that
$(\mu_n)_{n=1}^\infty $ satisfies the inequalities
\eqref{munballn} and that there are constants $C$ and $s$ such
that $\mu_n (B) \leq C^s$.

\begin{lemma}
  \label{lem:weaklimit}
  If $\tau \neq -l_1$ or $|\det A| > 1$, then the measures
  $\mu_n$ satisfy inequalities \eqref{munballn}.
\end{lemma}

\begin{proof}
  If $\tau < -l_1$, then $T^{-n} (B(z_n, e^{-\tau n}))$ consists
  of one or several long ellipses that wrap around the torus, see
  Figure~\ref{fig:complicatedellipse}.  Because of
  Lemma~\ref{lem:algebraic}, the semi-axes have irrational
  directions and therefore $\mu_n$ converges weakly to the
  Lebesgue measure on the torus, as illustrated in
  Figure~\ref{fig:complicatedellipse}. This is so since each
  ellipse is a thin neighbourhood around a piece of a stable
  manifold of length $\delta = e^{-(l_1 + \tau)n}$. If we let
  $\nu_n$ be the normalised Lebesgue measure on such a piece
  $W_\delta^\mathrm{s}$ of stable manifold, then $\nu_n$
  converges weakly to the two-dimensional Lebesgue measure on
  $\mathbb{T}^2$ as $n \to \infty$. (Use, for instance, Weyl's
  equidistribution theorem on any horizontal or vertical section
  of $W_\delta^\mathrm{s}$, i.e.\
  $(\mathbb{T} \times \{y\}) \cap W_\delta^\mathrm{s}$ or
  $(\{x\} \times \mathbb{T}) \cap W_\delta^\mathrm{s}$.)  Since
  $\mu_n$ is an average of such measures, $\mu_n$ converges
  weakly to the two-dimensional Lebesgue measure on
  $\mathbb{T}^2$.

  If $\tau > -l_1$, then there are two cases. If $|\det A| = 1$,
  then the dimension formula in this case gives the Hausdorff
  dimension 0, and there is nothing to prove.  If $|\det A| > 1$,
  then $T^{-n} (B(z_n, e^{-\tau n}))$ consists of $e^{Ln}$
  exponentially small ellipses.  Given $B=B(x,r)$, note that for
  $n$ large enough,
  \begin{equation}\label{mun1}
    \begin{split}
      \mu_n(B)&=\sum_{i\colon B\cap
        R_{n,i}\ne\varnothing}\pi^{-1}e^{2n\tau}\mu(B\cap
      R_{n,i})\\
      &\asymp e^{2n\tau}\mu( R_{n,i})\#\{i\colon B\cap
      R_{n,i}\ne\varnothing\}\\
      &\asymp e^{-n(l_1+l_2)}\#\{B\cap T^{-n}z_n\}\asymp
      e^{-n(l_1+l_2)}\#\{A^nB\cap \mathbb{Z}^2\}\\
    \end{split}
  \end{equation}
  Now we estimate $\#\{A^nB\cap \mathbb{Z}^2\}$. The set $A^nB$
  is an ellipsoid with semi-axes $re^{nl_1}$ and $re^{l_2n}$.
  The direction of the longer semi-axis of $A^nB$ is irrational,
  denoted by $\theta$. Observe that
  \[
    \#\{A^nB\cap \mathbb{Z}^2\}\asymp\{1\le k\le
    re^{l_2n}:\{\theta k\}\in I_n\},
  \]
  where $I_n$ is an interval of length $re^{l_1n}$.  Since
  $\theta $ is a quadratic irrational number, $\theta $ is of
  constant type. Given $N\ge 1$, it follows from Theorem~\ref{MK}
  that $\frac{d_\theta(N)}{d_\theta'(N)}$ is uniformly bounded,
  which implies that
  \[
    d_\theta(N)\asymp d_\theta'(N)\asymp N^{-1}.
  \]
  From what is written above we see that
  \[
    \#\{1\le k\le re^{l_2n}:\{\theta k\}\in I_n\}\asymp
    \frac{|I_n|}{d_\theta(re^{l_2n})}\asymp re^{nl_1}re^{l_2n}.
  \]
  Recalling \eqref{mun1}, we conclude that for any ball $B$,
  \[
    \mu_n(B)\asymp r^2
  \]
  holds for $n$ large enough.

  We now consider the case $\tau=-l_1$ and $|\det A | > 1$. In
  this case, $T^{-n} (B(z_n, e^{-\tau n}))$ consists of $e^{nL}$
  long ellipses transversal in $B$.
  Therefore
  \begin{equation}\label{tauequall1}
    \begin{split}
      \mu_n(B)&\asymp e^{2n\tau} \sum_{i\colon B\cap
        R_{n,i}\ne\varnothing}\mu(B\cap R_{n,i})\\
      &\asymp e^{2n\tau} \#\{T^{-n}z_n\cap R\} re^{-n(\tau+l_2)},
    \end{split}
  \end{equation} 
  where $R$ is an ellipse with the same centre as $B$, the length of
  which is a constant, independent of $B$ and $n$, in stable
  direction, and comparable to $r+e^{-n(\tau+l_2)}$ in unstable
  direction. From the argument above, we obtain
  \[
    \#\{T^{-n}z_n\cap R\}\asymp re^{n(l_2-\tau)}.
  \]
  Combining this with \eqref{tauequall1}, we have
  $\mu_n(B)\asymp r^2$.
\end{proof}

\begin{lemma}
  \label{lem:ballestimate}
  There is a constant $C$ such that
  \[
    \mu_n (B(x,r)) \leq \left\{
      \begin{array}{ll}
        C r^{\frac{2 l_2}{\tau + l_2}}
        & \text{if } 0<\tau < -l_1 \\
        C r^{\min\{\frac{l_1+l_2}{\tau + l_1},\frac{2 l_2}{\tau +
        l_2}\}}
        & \text{if } \tau > -l_1 ,
      \end{array} \right.                                                                                                           
  \]                                                                                                                                  
  for all $n$, $x$ and $r$.
\end{lemma}

\begin{proof}
  Pick a point $x \in \mathbb{T}^2$ and $r > 0$. We want to give
  an estimate of the $\mu_n$-measure of the ball $B(x,r)$ for
  $n\ge1$. There are three cases to consider: $\tau < -l_1$,
  $-l_1 < \tau < \frac{l_2 - l_1}{2}$ and
  $\frac{l_2-l_1}{2} < \tau$.

  \begin{figure}
    \begin{center}
      \includegraphics[width=0.19\textwidth]{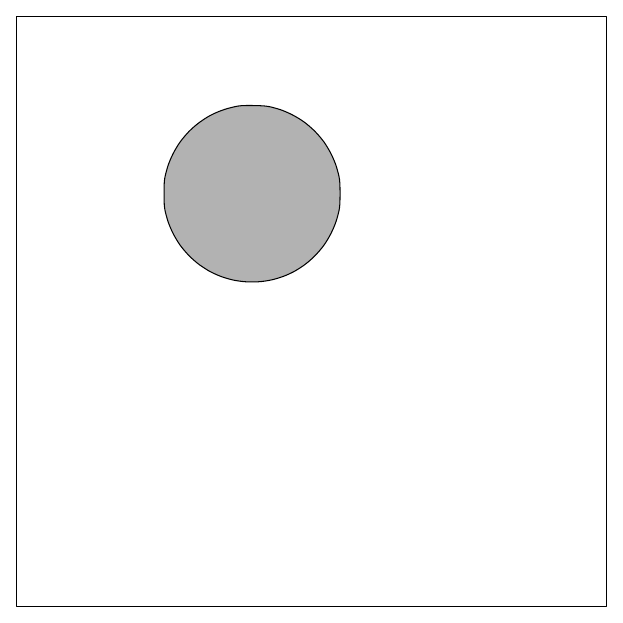}
      \includegraphics[width=0.19\textwidth]{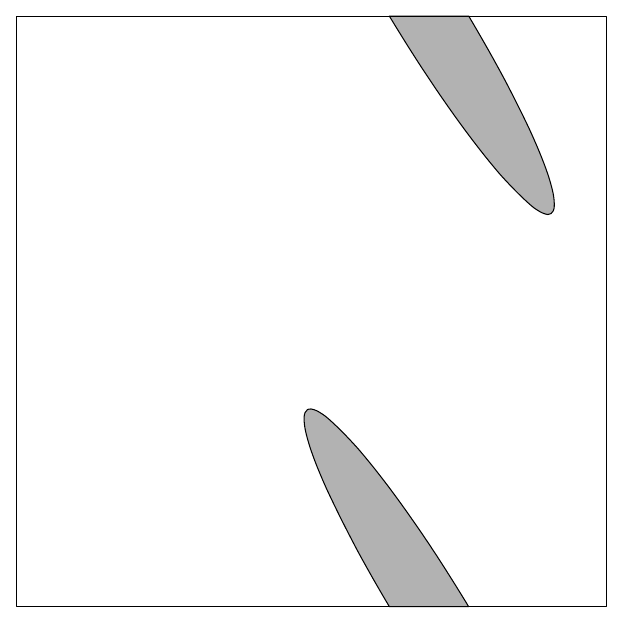}
      \includegraphics[width=0.19\textwidth]{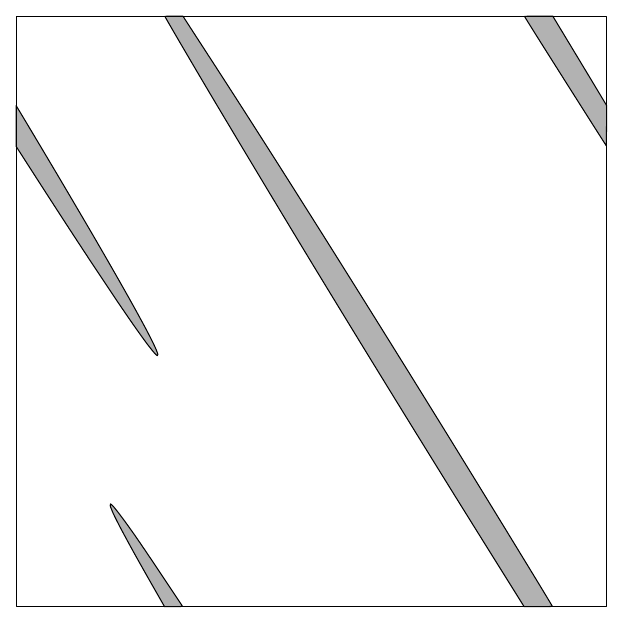}
      \includegraphics[width=0.19\textwidth]{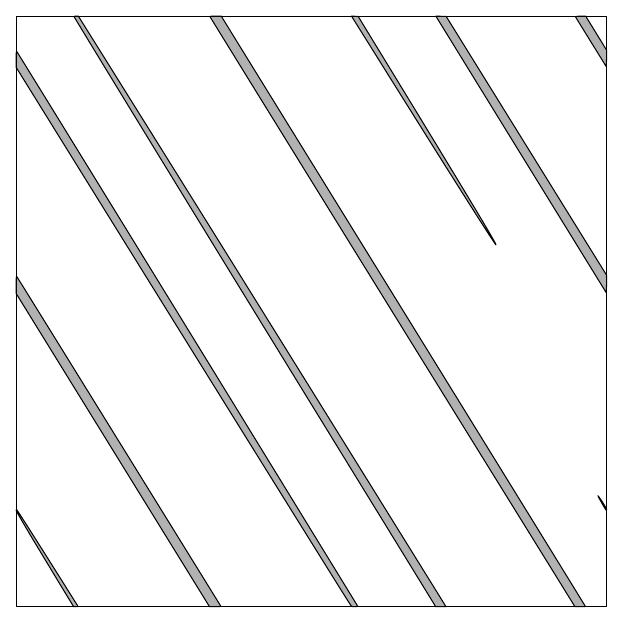}
      \includegraphics[width=0.19\textwidth]{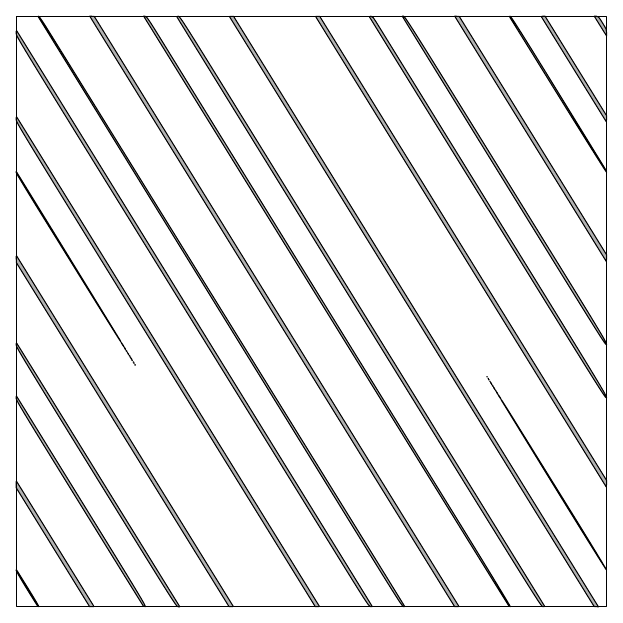}
    \end{center}
    \caption{Artist's illustration of the structure of
      $T^{-n} (B(z_n,e^{- \tau n}))$, for $n = 0, 1, 2, 3, 4$ and
      $A = \bigl( \protect\begin{smallmatrix}2 & 1 \protect\\ 1 &
        1 \protect\end{smallmatrix}
      \bigr)$.} \label{fig:complicatedellipse}
  \end{figure}

  We consider three cases, depending on the size of $\tau$.
  
  \begin{itemize}
  \item Case $\tau < -l_1$.

    In this case, we have $l_1 + \tau < 0 < \tau + l_2$.  Then on
    the torus, each ellipse in $T^{-n} (B(z_n,e^{- \tau n}))$ is
    very long, and hence wraps around the torus in a complicated
    way, see Figure~\ref{fig:complicatedellipse}.  To get a
    good estimate on $\mu_n(B(x,r))$, we need to investigate
    how the strips are distributed on the torus. If they are too
    concentrated, then $\mu_n(B(x,r))$ can be very large.

    First we estimate the distance between different parts of one
    ellipse. Consider a line in the unstable direction going out
    from the point $z_n = (z_{n,1}, z_{n,2})$. The equation of
    this line can be written as
    $y - z_{n,2} = \alpha (x - z_{n,1})$, and by
    Lemma~\ref{lem:algebraic}, the number $\alpha$ is an
    algebraic number of degree $2$. This line will wrap around
    the torus and come close to the point $z_n$.

    Suppose that $p$ and $q$ are integers such that for
    $x = q - z_{n,1}$ the point $(x,y)$ where
    $y = z_{n,2} + \alpha (x - z_{n,1})$ is very close to $z_n$,
    say $y = z_{n,2} + p \pm r$ with $0 < r < c e^{-\tau n}$. We
    then have $|\alpha q - p| < r$, or equivalently
    $|\alpha - \frac{p}{q}| < \frac{r}{q}$. Liouville's theorem
    on Diophantine approximation implies that
    $|\alpha - \frac{p}{q}| > \frac{c}{lq^2}$ and hence
    $q \geq ce^{\tau n}$.

    This implies that the rectangle around $B(z_n,e^{-\tau n})$
    with side-length $2 e^{-\tau n}$ in the stable direction and
    $c e^{\tau n}$ in the unstable direction, does not overlap
    itself, as illustrated in Figure~\ref{fig:rectangle}. Hence,
    any ellipse in $T^{-n}B(z_n, e^{-\tau n})$ does not overlap
    itself and the ellipse has a side-length
    $c e^{(\tau n - l_2)n}$ in the unstable direction, which
    means that the separation of the strips in such ellipse is at
    least $c e^{(\tau n - l_2)n}$, for some $c > 0$.

    \begin{figure}
      \begin{center}
        \includegraphics[width=0.3\textwidth]{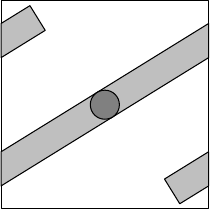}
      \end{center}
      \caption{Illustration of the rectangle which does not
        overlap itself.}
      \label{fig:rectangle}
    \end{figure}
    
    Consider again the rectangle $R$ around $B(z_n,e^{-\tau n})$
    with side-lengths $2 e^{-\tau n}$ in the stable direction and
    $c e^{\tau n}$ in the unstable direction. As we have seen
    above, the rectangle $R$ does not overlap itself. The
    preimage $T^{-n} (R)$ consists of $e^{Ln}$ pieces, but they
    must all be disjoint, since if two of them intersect, then
    their images under $T^n$, which both are $R$, must
    intersect. Since $R$ does not self-intersect, all pieces in
    $T^{-n} (R)$ are disjoint.

    This implies that we have the separation
    $c e^{(\tau - l_2)n}$ between the strips in
    $T^{-n} (B(z_n,e^{-\tau n}))$.
    
    Now we estimate $\mu_n(B)$. There are two cases, depending on
    how large $r$ is compared to the size of the ellipses that
    make up the set
    $T^{-n} (B(z_n,e^{-\tau n}))$.

    \begin{enumerate}
    \item[(i)] Assume that $r \leq e^{-(\tau - l_2)n}$. Then
      $B(x,r)$ intersects at most one strip. We therefore have
      that
      \[
        \mu_n (B(x,r)) \leq c r^{\frac{2l_2}{\tau + l_2}}.
      \]

    \item[(ii)] Suppose that $e^{-(\tau - l_2)} < r < 1$. Then
      $B(x,r)$ intersects at most $c r e^{(l_2 - \tau)n}$
      strips. Each strip gives a contribution to the measure
      which is $e^{2 \tau n} r e^{-(\tau+l_2)n}$. Hence
      \[
        \mu_n (B(x,r)) \leq c r^2 e^{(l_2 - \tau)n} e^{2 \tau n}
        r e^{-(\tau+l_2)n} = c r^2.
      \]
    \end{enumerate}

    Taken together, the two estimates imply that there is a
    constant $C$ such that
    $\mu_n (B(x,r)) \leq C r^{\frac{2 l_2}{\tau + l_2}}$ holds
    for all $n$, $x$ and $r$.

  \item Case $\tau \geq \frac{l_2-l_1}{2}$.

    In this case, we have
    $\tau + l_2 > \tau + l_1 \geq \frac{1}{2} (l_1 + l_2) \geq
    0$. Since we do not consider the case $\tau = -l_1$ in this
    lemma, we actually have $\tau + l_2 > \tau + l_1 > 0$ in this
    case. Note that $B(z_n,r_n)$ is contained in a ``rectangle''
    $R$ centred at $z_n$ and with side-lengths
    $ce^{\frac{1}{2}(l_2-l_1)n}$ in unstable direction and
    $e^{\frac{1}{2}(l_1-l_2)n}$ in stable direction. From what
    was written in the previous case, we conclude that such a
    rectangle doesn't intersect itself. The rectangles in the
    preimage $T^{-n}R$ has a side-length
    $ce^{-\frac{1}{2}(l_1+l_2)n}$ in unstable direction and
    $e^{-\frac{1}{2}(l_1+l_2)n}$ in stable direction and each
    such rectangle contain one of the ellipses in
    $T^{-n} (B(z_n,e^{-\tau n}))$.  Therefore the ellipses are
    separated by at least about $ce^{-\frac{1}{2}(l_1+l_2)n}$.

    We now bound the measure $\mu_n(B(x,r))$.

    \begin{enumerate}
    \item[(i)] If $r < e^{-(\tau + l_2)n}$, then as before
      \[
        \mu_n (B(x,r)) \leq r^{\frac{2l_2}{\tau + l_2}}.
      \]

    \item[(ii)] If $e^{-(\tau + l_2)n} < r < e^{-(\tau + l_1)n}$,
      the ball $B(x,r)$ intersects at most one ellipse, but can
      not contain an ellipse. We then have
      \[
        \mu_n(B(x,r)) \leq e^{2\tau n} r e^{-(\tau+l_2)n} = r
        e^{(\tau - l_2)n}.
      \]
      Since $\tau>\frac{l_2-l_1}{2}$, we have
      $\tau + l_1 > \frac{l_1+l_2}{2} \geq 0$. Therefore,
      $l_1+l_2 \geq 0$ implies that $\tau - l_2 \leq \tau + l_1$
      and hence that $\frac{\tau - l_2}{\tau + l_1} \leq
      0$. Because of this, we can use $r < e^{-(\tau + l_1)n}$ to
      estimate that
      \begin{align*}
        \mu_n (B(x,r))
        & \leq r e^{(\tau - l_2)n} = r \bigl( e^{-(\tau +
          l_1)n} \bigr)^{- \frac{\tau - l_2}{\tau + l_1}} \\
        & \leq r^{1 -
          \frac{\tau-l_2}{\tau + l_1}} = r^{\frac{l_1+l_2}{\tau +
          l_1}}.
      \end{align*}

    \item[(iii)] If
      $e^{-(\tau + l_1)n} < r < e^{-\frac{1}{2}(l_1+l_2)n}$, in
      this case $B(x,r)$ intersect at most one ellipse, and can
      contain an ellipse. Hence
      \[
        \mu_n(B)\le e^{-n(l_1+l_2)}<r^{\frac{l_1+l_2}{\tau+l_1}}.
      \]

    \item[(iv)] When $e^{-\frac{1}{2}(l_1+l_2)n}< r $, we have
      that $B(x,r)$ intersects at most
      $c re^{\frac{1}{2}(l_1+l_2)n}$ ellipses. Then
      \[
        \mu_n(B) \le re^{\frac{1}{2}(l_1+l_2)n} e^{2\tau n}
        e^{-(\tau+l_1)n} e^{-(\tau+l_2)n} = c r
        e^{-\frac{1}{2}(l_1+l_2)n} = c r^2.
      \]
    \end{enumerate}

    From above, we have
    $\frac{\log\mu_n(B(x,r))}{\log r} \geq
    \min\{\frac{l_1+l_2}{\tau + l_1},\frac{2l_2}{\tau +
      l_2},2\}=\frac{l_1+l_2}{\tau + l_1}$.  Hence,
    $\mu_n (B(x,r)) \leq C r^\frac{l_1+l_2}{\tau + l_1}$.

  \item Case $-l_1 \le \tau < \frac{l_2 - l_1}{2}$. 

    Here we can use the separation $e^{(\tau - l_2)n}$ between
    the ellipses in the unstable direction. We consider three
    cases.

    \begin{enumerate}
    \item[(i)] When $r \leq e^{-(\tau + l_2)n}$, the ball
      $B(x,r)$ intersects at most one ellipse, and we have
      \[
        \mu(B(x,r)) \leq r^2 e^{2\tau n} \leq r^2 r^{-
          \frac{2\tau}{\tau + l_2}} = r^{\frac{2l_2}{\tau + l_2}}.
      \]
      
    \item[(ii)] When
      $e^{-(\tau + l_2)n} < r < e^{(\tau - l_2)n}$, here it is
      important to note that
      $\tau - l_2 < - \frac{l_1 + l_2}{2} < 0$. Therefore
      \[
        e^{(\tau - l_2)n} = \bigl( e^{-(\tau + l_2)n} \bigr)^{-
          \frac{\tau - l_2}{\tau + l_2}} \leq r^{- \frac{\tau -
            l_2}{\tau + l_2}},
      \]
      since $- \frac{\tau - l_2}{\tau + l_2} > 0$.
  
      The ball $B(x,r)$ intersects at most one ellipse in the
      unstable direction. Hence
      \begin{align*}
        \mu (B(x,r))
        & \leq e^{2 \tau n} r e^{-(\tau + l_2)n} = r
          e^{(\tau - l_2)n} \\
        & \leq r^{1 - \frac{\tau - l_2}{\tau +
          l_2}} = r^{\frac{2l_2}{\tau + l_2}}.
      \end{align*}

    \item[(iii)] When
      $e^{-(\tau + l_1)n} < r < e^{(\tau - l_2)n}$, here with the
      same reason as (ii), the ball $B(x,r)$ also intersects at
      most one ellipse. Hence
      \begin{align*}
        \mu (B(x,r))
        & \leq e^{2 \tau n} e^{-(\tau + l_1)n}  e^{-(\tau + l_2)n} = 
          e^{-(l_1+ l_2)n} \\
        & \leq  r^{\frac{l_1+l_2}{\tau + l_1}}.
      \end{align*}
      
    \item[(iv)] When $r > e^{(\tau - l_2)n}$, then the ball
      $B(x,r)$ intersects at most about $r / e^{(\tau - l_2)n}$
      ellipses in the unstable direction, and we get
      \[
        \mu (B(x,r)) \leq e^{2 \tau n} r e^{-(\tau + l_2)n}
        \frac{r}{e^{(\tau - l_2)n}} = r^2.
      \]
    \end{enumerate}
    
    Combining (i)--(iii) above, we have
    \[
      \mu_n(B)\le C r^{\min\{\frac{2l_2}{\tau+l_2}, \frac{l_1+l_2}{\tau+l_1}\}}.
      \qedhere
    \]
  \end{itemize}
\end{proof}

By Lemma~\ref{lem:weaklimit} and \ref{lem:ballestimate}, the
assumptions of Lemma~\ref{lemmaforlb} are satisfied, and
therefore the set
$W_{\tau}$ has a large intersection property in sense that
$W_{\tau}\in\mathscr{G}^{s_{\tau}}$, where
\[
  s_{\tau}=\left\{
    \begin{array}{ll} \dfrac{2 l_2}{\tau + l_2} & 0 <
                                                  \tau < -l_1 \\
      \min\Bigl\{ \dfrac{l_1+ l_2}{\tau + l_1} ,\dfrac{2
      l_2}{\tau + l_2} \Bigr\} & \tau > -l_1 \rule{0pt}{22pt}
    \end{array} \right.    
\]
This finishes the proof of Theorem~\ref{the:dimensionofW1}.



\section{Proof of Theorem~\ref{the:dimensionofW2}}
\label{sec:lowerbound}

The proof of Theorem~\ref{the:dimensionofW2} follows some ideas from
Wang and Wu \cite{WangWu}. The following proposition is a
simplification of the corresponding statement in their paper
\cite[Proposition~3.1]{WangWu}.

\begin{proposition}\label{simple}
  Let $\tilde{s}_\tau$ be as in
  Theorem~\ref{the:dimensionofW2}. For $\tau\ge0$, we have
  \begin{align*}
    \tilde{s}_\tau
    &=\min_{t \in \{\, l_i+\tau : 1\le i\le d\,\}} \Biggl\{
       \sum_{j \in \mathcal{K}_1(t) \cup \mathcal{K}_2(t) } 1 +
       \frac{1}{t} \Biggl( \sum_{j\in \mathcal{K}_3(t)} l_j -
       \sum_{j \in \mathcal{K}_2(t)} \tau \Biggr) \Biggr\}\\
     &=\min_{t \in \mathcal{A}} \Biggl \{\sum_{j \in
       \mathcal{K}_1(t) \cup \mathcal{K}_2 (t) } 1 + \frac{1}{t}
       \Biggl( \sum_{j\in \mathcal{K}_3(t)} l_j - \sum_{j \in
       \mathcal{K}_2(t)} \tau \Biggr) \Biggr\}
  \end{align*}
  where
  \[
    \mathcal{A} = \{\, l_i,\,l_i+\tau : 1\le i\le d \,\}
  \]
  and for each $t\in \mathcal{A}$, the sets $\mathcal{K}_1(t)$,
  $\mathcal{K}_2(t)$, $\mathcal{K}_3(t)$ give a partition of
  $\{1,2,\ldots,d\}$ defined by
  \begin{align*}
    \mathcal{K}_1(t) &= \{\, j : l_j\ge t \,\},\\
    \mathcal{K}_2(t) &= \{\, j : l_j+\tau\le t \,\}\setminus
                       \mathcal{K}_1(t), \\
    \mathcal{K}_3(t) &= \{1,2,\ldots,d\}\setminus
                       (\mathcal{K}_1(t)\cup\mathcal{K}_2(t) ).
  \end{align*}
\end{proposition}

\begin{proof}
  The second equality holds due to Proposition~3.1 in
  \cite{WangWu}.

  Since $l_1\le l_2\le \ldots \le l_d$, for any $1\le i\le d$,
  assume that there is a $0\le k\le d-i$ such that
  $l_i = \ldots = l_{i+k} < l_{i+k+1}$. Here we adopt
  $l_{d+1}=\infty$. Then
  \begin{align*}
    \mathcal{K}_1(l_i+\tau) &= \{\, j : l_j\ge
                              l_i+\tau \, \} = \mathcal{W}(i),\\ 
    \mathcal{K}_2(l_i+\tau) &= \{1,2,\ldots,i,\ldots , i+k\}.
  \end{align*}
  Note that 
  \begin{align*}
    &\quad ~\frac{1}{\tau+l_i}\Bigl( \sum_{j\in \mathcal{K}_1
      (l_i+\tau) \cup\mathcal{K}_2 (l_i + \tau) } \tau + l_i +
      \sum_{j\in \mathcal{K}_3 (l_i + \tau)} l_j - \sum_{j \in
      \mathcal{K}_2 (l_i + \tau)} \tau \Bigr)\\
    &= \frac{1}{\tau+l_i} \Bigl( \sum_{j \in \mathcal{W}(i)}
      (l_i + \tau - l_j) + \sum_{j \in \mathcal{K}_1 (l_i + \tau)
      \cup \mathcal{K}_3 (l_i + \tau)} l_j + \sum_{j\in
      \mathcal{K}_2 (l_i + \tau)} l_i \Bigr)\\
    &= \frac{1}{\tau + l_i} \Bigl( \sum_{j\in \mathcal{W}(i)}
      (l_i + \tau - l_j) + (i+k) l_i + \sum_{j=i+k+1}^d l_j
      \Bigr) \\ 
    &= \frac{1}{\tau + l_i} \Bigl (\sum_{j\in \mathcal{W}(i)}
      (l_i  + \tau - l_j) + il_i + \sum_{j=i+1}^d l_j \Bigr).
  \end{align*}
  The last equality follows from the assumption on $k$. This
  finishes the proof.
\end{proof}

\begin{proof}[Proof of Theorem~\ref{the:dimensionofW2}]
  Without loss of generality, we assume that $\tau>0$.  Recall
  that $L=\sum l_k$, and put $W_n=T^{-n}B(z_n,e^{-\tau n})$, and
  $W_n$ is the union of $e^{Ln}$ ellipsoids, denoted by
  $\{R_n^k\}_{k=1}^{e^{Ln}}$. The ellipsoids in the set $W_n$
  have semi-axes $e^{-(\tau+l_k)n}$, $1\le k\le d$ which are all
  small. The separation of the ellipsoids is $e^{-nl_k}$,
  $1\le k\le d$ in the direction of the $d$ semi-axes.  Let
  \[
    \mu_n=\frac{\mu|_{W_n}}{\mu(W_n)}=c_de^{nd\tau}\mu|_{W_n},
  \]
  where $\mu$ is the Lebesgue measure. Let $B:=B(x,r)$, where
  $x\in \mathbb{T}^d$ and $r>0$, then one has
  \[
    \mu_n(B)=c_de^{n\tau}\mu(B\cap
    W_n)=c_de^{nd\tau}\sum_{\substack{ 1\le
      k\le e^{Ln}\\ B\cap R_n^k\ne\varnothing}}\mu(B\cap
    R_n^k).
  \]
  Now we show that $\mu_n(B(x,r))\le r^s$. We consider three
  cases, depending on the size of $r$.

  \begin{enumerate}
  \item[(i)] $r<e^{-n(\tau+l_d)}$.\\
    In this case, a ball of radius $r$ intersects only one
    ellipsoid of $W_n$. So
    \begin{align*}
      \mu_n(B) &\le c_de^{nd\tau} \sum_{\substack{1\le k\le e^{Ln} \\ B
      \cap R_n^k \ne \varnothing}} \mu(B)\\
               & \lesssim e^{nd\tau} r^d <
                 r^{d-\frac{d\tau}{\tau+l_d}} =
                 r^{\frac{dl_d}{\tau+l_d}}.
    \end{align*}

  \item[(ii)] $r\ge e^{-nl_1}$.\\
    The ball $B$ intersects at most
    \[
      \prod_{k=1}^dre^{l_kn}=r^de^{Ln}
    \]
    ellipsoid of  $W_n$. Thus
    \begin{equation*}
      \mu_n(B) \le c_d e^{nd\tau} \sum_{\substack{1\le k\le e^{Ln} \\
          B\cap R_n^k \ne \varnothing}} \mu(R_n^k) \lesssim
      e^{nd\tau} r^d e^{Ln} \frac{e^{-nd\tau}}{e^{Ln}} = r^d.
    \end{equation*}

  \item[(iii)] $e^{-n(\tau+l_d)}\le r< e^{-nl_1}$.\\
    In this case, the ball $B$ intersects many ellipsoids, and
    each ellipsoid is in the worst case transversal in $B$.  The
    ellipsoid in $W_n$ are parallel and each ellipsoid has $d$
    semi-axes. In each of these $d$ directions, we will estimate
    the number of ellipsoid segments which a ball of radius $r$
    intersects, aiming to get the total number of
    ellipsoid segments intersecting $B$.

    Recall that 
    \[
      \mathcal{A} = \{\, l_i,l_i+\tau : 1\le i\le d \, \}.
    \]
    Arrange the elements in $\mathcal{A}$ in non-descending
    order, and assume that there exists $i$ such that
    \[
      e^{-nt_{i+1}}\le r<e^{-nt_i},
    \]
    where $t_i,~t_{i+1}$ are two consecutive and distinct terms
    in $\mathcal{A}$.  For such $i$, put
    \begin{align*}
      N_1(i) &:= \{\, j : l_j>t_i \,\} = \{ \, j :
              e^{-nl_j}<e^{-nt_i} \,\}, \\ 
      N_2(i) &:= \{\, j : l_j+\tau \le t_i \,\} = \{ \, j :
              e^{-n(l_j+\tau)} \ge e^{-nt_i} \,\},
    \end{align*}
    and
    \[
      N_3(i):=\{1,2,\ldots,d\}\setminus (N_1(i)\cup N_2(i)).
    \]
    Since $t_i,~t_{i+1}$ are consecutive, we have 
    \begin{equation}\label{ni1}
      N_1(i) = \{\, j : l_j\ge t_{i+1} \,\},
    \end{equation}
    and 
    \begin{equation}\label{ni2}
      N_2(i) = \{\, j : t_{i+1}>l_j+\tau \,\}.
    \end{equation}

    \begin{enumerate}
    \item[(a)] For $j\in N_1(i)$, by equation \eqref{ni1}, one
      has $e^{-nl_j}\le r$. Hence in each of these directions,
      the number of ellipsoids intersecting a ball with radius
      $r$ is at most
      \[
        \frac{r}{e^{-nl_j}}.
      \]
 
    \item[(b)] For $j\in N_2(i)$, we have
      $r\le e^{-n(l_j+\tau)}$. In the direction of this
      semi-axis, a ball with radius $r$ intersects at most one
      ellipse of $W_n$.

    \item[(c)] For $j\in N_3(i)$, It follows from the definition
      of $N_3(i)$ and equation \eqref{ni2} that
      \[
        e^{-n(l_j+\tau)}\le e^{-nt_{i+1}} < e^{-nt_i} \le
        e^{-nl_j},
      \]
      which implies that
      \[
        e^{-n(l_j+\tau)}\le r< e^{-nl_j}.
      \]
      In this direction, because of the assumption on $r$ and the
      separation of the ellipsoids, the ball intersects at most
      one ellipsoid.
    \end{enumerate}

    From above, the ball  $B$  intersects at most 
    \[
      \prod_{j\in N_1(i)}\frac{r}{e^{-nl_j}}\prod_{j\in
        N_2(i)\cup N_3(i)}1 = \prod_{j\in N_{1}(i)} \frac{r}{e^{-nl_j}}
    \]
    ellipsoids, and the intersection of $B$ and each of these
    ellipsoids is contained in a rectangle of sides about
    \[
      L_j=\left\{
        \begin{array}{ll}
          2e^{-n(\tau+l_j)} & \text{if  \ }j\in N_1(i)\cup
                              N_3(i)\,,\\[2ex]
          2r & \text{if \ }j\in N_2(i) \, .
        \end{array}
      \right.
    \]
    It follows that 
    \begin{align}
      \mu_n(B) &= c_d e^{nd\tau} \sum_{\substack{1\le k\le e^{Ln} \\ B
      \cap R_n^k \ne \varnothing}} \mu(R_n^k\cap B)
      \nonumber \\ 
               & \lesssim e^{nd\tau} \prod_{j\in N_1(i)}
                 e^{-n(\tau+l_j)} \frac{r}{e^{-nl_j}} \prod_{j\in
                 N_2(i)}r \prod_{j\in N_3(i)} e^{-n(\tau+l_j)}
                 \nonumber \\
               & = e^{n(\sum_{j\in N_2(i)} \tau - \sum_{j\in
                 N_3(i)}l_j)} \prod_{j\in N_1(i)\cup N_2(i)}r.
                 \label{munb3}
    \end{align}
  \end{enumerate}
  
  The estimate obtained in case (iii) above will now be discussed in
  two cases.

  \begin{enumerate}
  \item[(a)] If
    $\sum_{j\in N_2(i)}\tau-\sum_{j\in N_3(i)}l_j\ge 0$, then
    using $r<e^{-nt_i}$, the inequalities \eqref{munb3} will be
    rewritten as
    \[
      \mu_n(B)\le r^{s_{i,1}},
    \]
    where
    \[
      s_{i,1} \colon = \sum_{j\in N_1(i)\cup N_2(i)} 1 -
      \frac{1}{t_i} \Bigl(\sum_{j\in N_2(i)}\tau-\sum_{j\in
        N_3(i)}l_j \Bigr).
    \]
    Notice that
    \begin{align*}
      \mathcal{K}_1(t_i) &= N_1(i)\cup \{\, j : l_j=t_i \,\}, \\
      \mathcal{K}_2(t_i) &= N_2(i),\\
      \mathcal{K}_3(t_i) &= N_3(i)\setminus \{\, j :
                           l_j=t_i \,\}.
    \end{align*}
    Hence
    \begin{align}
      s_{i,1} &= \sum_{j\in \mathcal{K}_1 (t_i) \cup
                \mathcal{K}_2 (t_i)} 1 - \sum_{j :
                l_j=t_i}1 - \frac{1}{t_i} \bigl( \sum_{j\in
                \mathcal{K}_2(t_i)} \tau - \sum_{j\in
                \mathcal{K}_3 (t_i)} l_j - \sum_{j :
                l_j=t_i}l_j \bigr) \nonumber\\
              & = \sum_{j\in \mathcal{K}_1 (t_i) \cup
                \mathcal{K}_2 (t_i)} 1 - \frac{1}{t_i} \Bigl(
                \sum_{j\in \mathcal{K}_2(t_i)} \tau -
                \sum_{j\in \mathcal{K}_3(t_i)} l_j \Bigr).
                \label{si1}
    \end{align}
    This shows that $s_{i,1}$ is a term in
    \[
      \biggl\{ \sum_{j\in \mathcal{K}_1(t) \cup \mathcal{K}_2(t)
      } 1 + \frac{1}{t} \Bigl( \sum_{j\in \mathcal{K}_3(t)} l_j -
      \sum_{j \in \mathcal{K}_2(t)} \tau \Bigr) \biggr\}_{t \in
        \mathcal{A}},
    \]
    as given in Proposition~\ref{simple}.
    
  \item[(b)] If
    $\sum_{j\in N_2(i)} \tau - \sum_{j\in N_3(i)}l_j< 0$, by
    $r\ge e^{-nt_{i+1}}$ and inequalities \eqref{munb3}, then
    \[
      \mu_n(B)\le r^{s_{i,2}},
    \] 
    where
    \[
      s_{i,2}\colon =\sum_{j\in N_1(i)\cup
        N_2(i)}1-\frac{1}{t_{i+1}} \Bigl( \sum_{j\in
        N_2(i)}\tau-\sum_{j\in N_3(i)}l_j \Bigr).
    \]
    Notice that 
    \begin{align*}
      \mathcal{K}_1(t_{i+1}) & =N_1(i),\\
      \mathcal{K}_2(t_{i+1}) &= N_2(i)\cup  \{\, j :
                               l_j=t_{i+1} \,\},\\
      \mathcal{K}_3(t_{i+1}) &= N_3(i)\setminus \{\, j :
                               l_j=t_{i+1} \,\}.
    \end{align*}
    Hence
    \begin{align}
      s_{i,2} &= \sum_{j\in \mathcal{K}_1 (t_{i+1}) \cup
                \mathcal{K}_2 (t_{i+1})} 1 - \sum_{j :
                l_j=t_{i+1}} 1 - \nonumber \\
              & \hspace{2cm} - \frac{\sum_{j \in
                \mathcal{K}_2(t_{i+1})} \tau - \sum_{j \in
                \mathcal{K}_3(t_{i+1})} l_j - \sum_{j :
                l_j=t_{i+1}} l_j}{t_{i+1}} \nonumber\\
              &= \sum_{j \in \mathcal{K}_1 (t_{i+1}) \cup
                \mathcal{K}_2 (t_{i+1})} 1 -
                \frac{1}{t_{i+1}} \biggl(\sum_{j\in
                \mathcal{K}_2(t_{i+1})} \tau - \sum_{j\in
                \mathcal{K}_3 (t_{i+1})}l_j \biggr).
                  \label{si2}
    \end{align}
    Hence $s_{i,2}$ is also a term among those in
    Proposition~\ref{simple}.
  \end{enumerate}
    
  Let
  \[
    \tilde{s}(i)=\sum_{j\in \mathcal{K}_1(t_i)\cup
      \mathcal{K}_2(t_i)}1-\frac{1}{t_i} \biggl(\sum_{j\in
      \mathcal{K}_2(t_i)}\tau-\sum_{j\in
      \mathcal{K}_3(t_i)}l_j\biggr).
  \]
  Combining \eqref{si1} and \eqref{si2}, we get that
  \[
    \mu_n(B)\le r^{\min_{1\le i<d} \tilde{s}(i)}.
  \]
  
  Note that $\tilde{s}(d)=\frac{dl_d}{\tau+l_d}<d$, together
  with Cases (i), (ii) and Proposition~\ref{simple}, then we
  conclude that there is a constant $C$ such that
  \[
    \mu_n(B)\le C r^s,
  \]
  where 
  \[
    s=\min_{1\le j\le d} \biggl\{ \frac{jl_j+\sum_{i=j+1}l_i -
      (l_i - l_j - \tau)_+}{\tau+l_j} \biggr\}.
  \]
  
  Lemma~\ref{lemmaforlb} now finishes the proof, since $\mu_n$
  converges weakly to the Lebesgue measure on the torus if all
  eigenvalues are outside the unit circle.
\end{proof}

\section{Examples} \label{sec:examples}

Recall that by Theorems~\ref{the:dimensionofW2} and
\ref{upperboundofW},
\begin{align*}
  \min_k \Biggl\{ \frac{k l_ k + \sum_{j > k} l_j -
  \sum_{j = 1}^d (l_j - l_k - \tau)_+}{\tau + l_k} \Biggr\}
  & \leq \dimh W_\tau \\
  & \leq \min_k \left \{ \dfrac{kl_k + \sum_{j>k}
    l_j}{\tau + l_k} \right \},
\end{align*}
where the minimum in the upper bound is over those $k$ for which
$\tau+l_k>0$, and the lower bound is only valid if all
eigenvalues are outside the unit circle. We expect that the lower
bound holds in more cases than those given by
Theorem~\ref{the:dimensionofW2}, but this is not always the case,
which we will show in several examples.

Here we will consider tori of several different dimensions and
compare the corresponding sets $W_\tau$. To keep them apart, we
will use the notation $W_\tau(T)$, where
$T \colon \mathbb{T}^d \to \mathbb{T}^d$.

\begin{example}
  The first example is a transformation $T$ on $\mathbb{T}^4$ and
  another related transformation $S$ on $\mathbb{T}^2$. For any
  $m\in\mathbb{N}$, let
  \[
    A_m = \begin{pmatrix}m+1 & m \\ 1 & 1 \end{pmatrix}.
  \]
  Then, $\det A_m=1$ and $\tr A_m=m+2$. The eigenvalues of $A_m$
  are
  \[
    \lambda_{m,-} = \dfrac{m}{2} + 1 - \sqrt{\dfrac{m^2}{4} + m},
    \qquad \lambda_{m,+} = \dfrac{m}{2} + 1 +
    \sqrt{\dfrac{m^2}{4} + m}.
  \]
  Take $m>1$ and put
  \[
    A = \begin{pmatrix} 2 & 1 & 0 \\ 1 & 1 & 0 \\ 0 & 0 &
      A_m \end{pmatrix} = \begin{pmatrix} A_1 & 0 \\ 0 &
      A_m \end{pmatrix}.
  \]
  Thus, $A$ is an integer matrix with no eigenvalues on the unit
  circle and $\det A=1$. The transformation
  $T \colon \mathbb{T}^4 \to \mathbb{T}^4$ defined by
  \[
    T(x) = Ax \mod \quad 1
  \]
  have four Lyapunov exponents
  \[
    l_1 = \log \lambda_{m,-}, \quad l_2 = \log\dfrac{3 -
      \sqrt{5}}{2}, \quad l_3 = \log\dfrac{3 + \sqrt{5}}{2},
    \quad l_4 = \log\lambda_{m,+}.
  \]
  Therefore, $l_1<l_2<0<l_3<l_4$, $l_1+l_4=0$ and
  $l_1+l_2+l_3+l_4=0$. We will consider $W_\tau(T)$ and prove
  that we do not have equality in Theorem~\ref{upperboundofW} for
  $W_\tau(T)$.
  
  Consider the transformation
  $S \colon \mathbb{T}^2 \to \mathbb{T}^2$ defined by
  $S(x)=A_mx \mod 1$. The dimension
  formula in Theorem~\ref{the:dimensionofW1} tells us that
  \begin{equation*}
    \dimh W_\tau(S)=\left\{
      \begin{array}{ll}
        \dfrac{2l_4}{\tau+l_4}, & \tau\in[0,l_4), \\
        \dfrac{l_1+l_4}{\tau+l_1}=0, & \tau\in[l_4,\infty).
                                       \rule{0pt}{22pt}
      \end{array}\right.
  \end{equation*}

  We shall now compare with $W_\tau (T)$. In this case, the upper
  bound from Theorem~\ref{upperboundofW} says that
  \begin{equation*}
    \dimh W_\tau(T) \leq \left\{
      \begin{array}{ll}
        \dfrac{4 l_4}{\tau + l_4},
        & \tau \in [0,l_4/3), \\
        \dfrac{3l_3+l_4}{\tau+l_3}=0,
        & \tau\in \biggl[ \dfrac{l_4}{3},\dfrac{l_3 + l_4}{2}
          \biggr), \rule{0pt}{22pt} \\
        \dfrac{2l_2 + l_3 + l_4}{\tau + l_2} = \dfrac{l_2 +
        l_4}{\tau + l_2},
        & \tau \in \biggl[ \dfrac{l_3 +
          l_4}{2},l_2 + l_3 + l_4 \biggr), \rule{0pt}{22pt} \\
        \dfrac{l_1 + l_2 + l_3 + l_4}{\tau + l_1}=0,
        & \tau\in[l_2 + l_3 + l_4, \infty). \rule{0pt}{22pt}
    \end{array}\right.
  \end{equation*} 

  In fact, for $\tau \in \bigl(\frac{l_3+l_4}{2},l_4\bigr)$ and
  $z_n=0$, the set $E_n=T^{-n}B(0,e^{-n\tau})$ is a thin and long
  ellipsoid that is wrapped around the torus.  This ellipsoid's
  three shortest semi axes shrink with an exponential speed in
  $n$ and the largest grows. Therefore, we have
  \[
    W_\tau(T)=(0,0) \times W_\tau(S).
  \]
  Thus,
  \[
    \dimh W_\tau(T) = \dimh W_\tau(S) = \dfrac{2 l_4}{\tau +
      l_4},
  \]
  for this $\tau$.  The upper bound for this $\tau$ however is
  \[
    \dimh W_\tau(T) \leq \dfrac{l_2 + l_4}{\tau + l_2}.
  \]
  It is clear that
  \[
    \dfrac{l_2+l_4}{\tau+l_2}\neq\dfrac{2l_4}{\tau+l_4},
  \]
  unless $\tau = l_4$, and hence the upper bound is not the true
  value of the Hausdorff dimension in general.

  When $\tau$ is close to $l_4$, then it is clear that the upper
  and lower bounds from Theorems~\ref{the:dimensionofW2} and
  \ref{upperboundofW} coincide, and hence the lower bound is not
  valid in this case.

  Note that for the range of $\tau$ considered above, the measure
  $\mu_n$ does not converge weakly to the Lebesgue measure on
  $\mathbb{T}^4$, but it does converge weakly to the Lebesgue measure
  on $(0,0) \times \mathbb{T}^2$.
\end{example}

\begin{example}
  \label{ex:last}
  The second example is a transformation $T$ on
  $\mathbb{T}^3$. For any integer $m>2$,
  let
  \[
    A = \begin{pmatrix} m & 0 & 0 \\ 0 & 2 & 1 \\ 0 & 1 &
      1 \end{pmatrix}.
  \]
  The eigenvalues
  are
  \[
    m, \quad \dfrac{3+\sqrt{5}}{2}, \quad \dfrac{3-\sqrt{5}}{2}.
  \]
  Hence,
  \[
    l_1 = \log \dfrac{3 - \sqrt{5}}{2}, \quad l_2 =
    \log\dfrac{3+\sqrt{5}}{2}, \quad l_3 = \log m.
  \]
  Thus, $l_1<0<l_2<l_3$ and $l_1+l_2+l_3=l_3=\log m$. 
  
  Take points $(a_n,b_n,c_n) \in \mathbb{T}^3$ and consider
  $\tau > l_2$. Let $z_n$ be such that
  $z_n = T^n (a_n,b_n,c_n)$. Since $\tau > l_2$, we have that
  $T^{-n} (B(z_n, e^{-\tau n}))$ is contained in the set
  \[
    \mathbb{T} \times (b_n -
    e^{-(\tau-l_2)n} , b_n + e^{-(\tau-l_2)n}) \times (c_n -
    e^{-(\tau-l_2)n} , c_n + e^{-(\tau-l_2)n}).
  \]
  It is therefore possible to choose $(b_n,c_n)$ such that
  $W_\tau = \emptyset$. In particular, for such a choice of
  $(b_n,c_n)$ we have $\dimh W_\tau = 0$ for $\tau > l_2$.  In
  this case, the result of Hill and Velani
  \cite[Theorem~1]{HillVelaniMatrix} is not correct. Note also
  that this shows that the lower bound in
  Theorem~\ref{the:dimensionofW2} does not hold in general, since
  it gives a positive dimension in this case.

  As in the previous example, we note that for $\tau > l_2$, the
  measure $\mu_n$ does not converge weakly to the Lebesgue
  measure on $\mathbb{T}^3$, but if $(b_n,c_n) \to (b,c)$, then
  it converges weakly to the Lebesgue measure on
  $\mathbb{T} \times (b,c)$.
\end{example}

\bibliographystyle{amsplain}

\end{document}